\documentclass[default,iicol]{sn-jnl}

\usepackage{graphicx}%
\usepackage{multirow}%
\usepackage{amsmath,amssymb,amsfonts}%
\usepackage{amsthm}%
\usepackage{mathrsfs}%
\usepackage[title]{appendix}%
\usepackage{xcolor}%
\usepackage{textcomp}%
\usepackage{manyfoot}%
\usepackage{booktabs}%
\usepackage{algorithm}%
\usepackage{algorithmicx}%
\usepackage{algpseudocode}%
\usepackage{listings}%

\theoremstyle{thmstyleone}%
\newtheorem{theorem}{Theorem}
\newtheorem{proposition}[theorem]{Proposition}%
\newtheorem{lemma}[theorem]{Lemma}%

\theoremstyle{thmstyletwo}%

\theoremstyle{thmstylethree}%

\begin{document}

\title[On the dynamics and integrability of the Ziegler pendulum]{On the dynamics and integrability of the Ziegler pendulum}


\author*[1,2,3,4]{\fnm{Ivan Yu.} \sur{Polekhin}}\email{ivanpolekhin@mi-ras.ru}

\affil*[1]{\orgname{Steklov Mathematical Institute of Russian Academy of Sciences}, \orgaddress{\street{8 Gubkina st.}, \city{Moscow}, \postcode{119991}, \country{Russia}}}

\affil[2]{\orgname{Moscow Institute of Physics and Technology}, \orgaddress{\street{9 Institutskiy per.}, \city{Dolgoprudny, Moscow Region}, \postcode{141701}, \country{Russia}}}

\affil[3]{\orgname{Lomonosov Moscow State University}, \orgaddress{\street{GSP-1, Leninskie Gory}, \city{Moscow}, \postcode{119991}, \country{Russia}}}

\affil[4]{\orgname{P.G. Demidov Yaroslavl State University}, \orgaddress{\street{14 Sovetskaya st.}, \city{Yaroslavl}, \postcode{150003}, \country{Russia}}}


\abstract{We prove that the Ziegler pendulum --- a double pendulum with a follower force --- can be integrable, provided that the stiffness of the elastic spring located at the pivot point of the pendulum is zero and there is no friction in the system. We show that the integrability of the system follows from the existence of two-parameter families of periodic solutions. We explain \textcolor{black}{a} mechanism for the transition from integrable dynamics, \textcolor{black}{for which there exist two first integrals and solutions belong to two-dimensional tori in a four-dimensional phase space}, to more complicated dynamics. The case in which the stiffnesses of both springs are non-zero is briefly studied numerically. We show that regular dynamics coexists with chaotic dynamics.}




\maketitle

\section{Introduction}\label{sec1}

The Ziegler pendulum is a well-known non-conservative mechanical system with many real-life applications. This system is mostly known for the `paradox' of loss of stability in the presence of friction, which was discovered 70 years ago \cite{ziegler1952stabilitatskriterien} (see also \cite{herrmann1964stability,bolotin1969effects}). A more detailed treatment of this and related systems and the history of the problem can be found in \cite{kirillov2013nonconservative}. A complete list of references on the Zielger pendulum and other non-conservative systems would be too long to give here and we also refer the reader to book \cite{kirillov2013nonconservative}. However, we would like to note that, to the best of our knowledge, the dynamics \textit{in the large} and integrability of the Ziegler pendulum have not been studied in the literature so far. At the same time, there are many papers on the dynamics of the classical double pendulum, where it is shown experimentally \cite{shinbrot1992chaos}, numerically \cite{stachowiak2006numerical}, and analytically \cite{dullin1994melnikov} using the Poincar{\'e}-Melnikov method that this system is chaotic. Early non-trivial results on the non-integrability of systems in non-potential force fields have been obtained in \cite{kozlov2022integrability}. The problem of integrability of the Ziegler pendulum has been also formulated in \cite{kozlov2022integrability}.

In this paper we study the dynamics of a planar double pendulum in the presence of a follower force of constant magnitude. We assume that this follower force is acting along a rod that is not connected to the fixed pivot point. There are no gravity and friction forces in the system, yet there are two springs of linear stiffness located in the joints of the double pendulum.

The paper presents both analytical and numerical results. Employing an analytical approach we study the case in which  the stiffness of the spring located at the pivot point is zero. Since we have the rotational symmetry, we can consider the reduced system of the third order. We prove that in some cases (for instance, for small follower forces) in the reduced system there exists a \textcolor{black}{two-parameter} family of periodic solutions. It means that in some region of the phase space of the system there exist two first integrals of the system and one can consider this system to have regular dynamics, at least for some initial conditions. However, these periodic solutions do not cover the whole phase space and in some parts of the phase space the system may not be integrable. We explain the mechanism of loss of periodicity of solutions associated with the change of parameters or initial conditions and show how the birth of non-periodic solutions leads to the non-integrability of the system.

The case in which the stiffnesses of the springs are non-zero is studied numerically. In this case, we can also observe the coexistence of two types of solutions in the phase space: regular solutions, which lie on two-dimensional surfaces, and chaotic trajectories. We show this based on the numerically obtained Poincar{\'e} sections and Lyapunov characteristic exponents of some solutions of the system.

\section{Equations of motion}\label{sec2}

Let us consider a planar double mathematical pendulum. One of its absolutely rigid massless rod of length $l_2$ can rotate freely in the plane about a fixed point $O$. Two other rigid massless rods of lengths $l_1$ and $l_3$ are connected to the free end $A$ of the first rod $OA$. Let us denote the ends of these rods by $B$ and $C$, respectively (Fig.~\ref{fig1}). We assume that points $A$, $B$, and $C$ always remain on a straight line {during the motion of the pendulum.} There is a point mass $m_2$ at point $A$ and point masses $m_1$ and $m_3$ at points $B$ and $C$, respectively. There is a force $F$ of constant magnitude acting along segment $BC$. Two springs of stiffnesses $k_1$ and $k_2$ are located at points $A$ and $O$, respectively. 

\begin{figure*}[t]
   \centering
   \includegraphics[width=\linewidth]{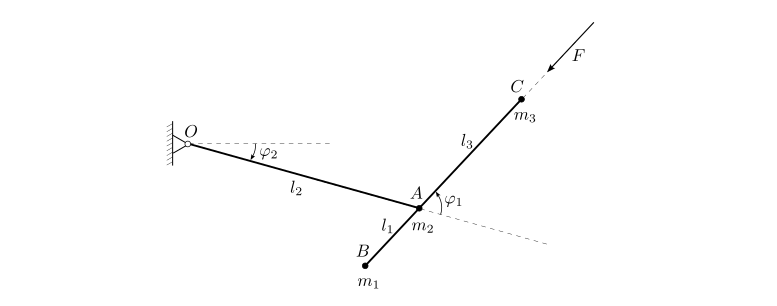}
   \caption{The Ziegler pendulum.}
   \label{fig1}
\end{figure*}

Let $\varphi_2$ be a generalized coordinate, the angle between a constant direction in the plane and segment $OA$, and $\varphi_1$ be the angle between $OA$ and $BC$. The kinetic energy of the system can be expressed in these coordinates as follows:
\begin{align}
\label{eq1}
\begin{split}
&T = \frac{m_2}{2} \dot\varphi_2^2 l_2^2 + \\
&\frac{m_1}{2} \left( \dot \varphi_2^2 l_2^2 + (\dot \varphi_1 + \dot \varphi_2)^2 l_1^2 + 2 l_1 l_2 \dot \varphi_2(\dot \varphi_1 + \dot \varphi_2) \cos \varphi_1 \right) +\\
& \frac{m_3}{2}\left( \dot\varphi_2^2 l_2^2 + (\dot\varphi_1 + \dot\varphi_2)^2 l_3^2 - 2l_2 l_3 \dot \varphi_2 (\dot \varphi_1 + \dot \varphi_2) \cos \varphi_1 \right).
\end{split}
\end{align} 
The potential energy is
\begin{align}
\label{eq2}
\begin{split}
\Pi = \frac{k_1 \varphi_1^2}{2} + \frac{k_2 \varphi_2^2}{2}.
\end{split}
\end{align}
The components of generalized force are
\begin{align}
\label{eq3}
\begin{split}
Q_1 = 0, \quad Q_2 = -Fl_2 \sin \varphi_1.
\end{split}
\end{align}
From \eqref{eq1}, \eqref{eq2}, and \eqref{eq3} we obtain the Lagrange equations, which can be represented as
\begin{align}
\label{eq5}
\begin{split}
& A_{11} \ddot \varphi_1 + A_{12} \ddot \varphi_2 = r_1,\\
& A_{21} \ddot \varphi_1 + A_{22} \ddot \varphi_2 = r_2,
\end{split}
\end{align}
where
\begin{align*}
\begin{split}
A_{11} & = m_1l_1^2 + m_3l_3^2,\\
A_{12} & = A_{21} = m_1 l_1^2 + m_3 l_3^2 + m_1 l_1 l_2 \cos \varphi_1\\
& - m_3 l_2 l_3 \cos\varphi_1,\\
A_{22} & = m_1 l_2^2 + m_1 l_1^2 + m_2 l_2^2 + m_3 l_2^2 + m_3 l_3^2\\
& + 2 m_1 l_1 l_2 \cos \varphi_1 - 2 m_3 l_2 l_3 \cos \varphi_1,\\
r_1 & = - k_1 \varphi_1 - m_1 l_1 l_2 \dot\varphi_2^2 \sin \varphi_1 + m_3 l_2 l_3 \dot \varphi_2^2 \sin \varphi_1,\\
r_2 & = -F l_2 \sin \varphi_1 - k_2 \varphi_2 + m_1 l_1 l_2 \dot \varphi_1 (\dot \varphi_1 + 2 \dot \varphi_2) \sin \varphi_1\\
& - m_3 l_2 l_3 \dot \varphi_1 (\dot \varphi_1 + 2 \dot \varphi_2) \sin \varphi_1.
\end{split}
\end{align*}
If $k_2 = 0$, then the right-hand side does not depend on {the} variable $\varphi_2$ and system \eqref{eq5} reduces to the three-dimensional system with the state variables $\varphi_1$, $v_1=\dot{\varphi}_1$, and $v_2=\dot{\varphi}_2$:
\begin{align}
\label{eq4}
\begin{split}
& \dot \varphi_1 = v_1,\\
& \dot v_1 = \left( r_1 - \frac{A_{12}}{A_{22}} r_2 \right)\left( A_{11} - \frac{A_{12} A_{21}}{A_{22}} \right)^{-1},\\
& \dot v_2 = \left( r_2 - \frac{A_{21}}{A_{11}} r_1 \right)\left( A_{22} - \frac{A_{12} A_{21}}{A_{11}} \right)^{-1}.
\end{split}
\end{align}

\begin{lemma}
{System \eqref{eq5} has a smooth invariant measure.}
\end{lemma}
\begin{proof}
{The components of generalized force \eqref{eq3} do not depend on the generalized velocities $\dot \varphi_1$ and $\dot \varphi_2$. Therefore, the Lagrange equations have the following form:
$$
\frac{d}{dt}\frac{\partial L}{\partial \dot \varphi_i} - \frac{\partial L}{\partial \varphi_i} = Q_i(\varphi_1, \varphi_2), \quad i = 1,2.
$$
After the Legendre transformation, $H(p_1, p_2, \varphi_1, \varphi_2) = \sum\limits_{i=1}^2 p_i \dot q_i - L(\dot \varphi_1, \dot \varphi_2, \varphi_1, \varphi_2)$, where $p_i = \frac{\partial L}{\partial \dot \varphi_i}$, $i=1,2$, this system can be rewritten as
$$
\dot \varphi_i = \frac{\partial H}{\partial p_i}, \quad \dot p_i = -\frac{\partial H}{\partial \varphi_i} + Q_i(\varphi_1, \varphi_2), \quad i=1,2.
$$
Note that the change of variables $(\dot \varphi_1, \dot \varphi_2, \varphi_1, \varphi_2) \mapsto (p_1, p_2, \varphi_1, \varphi_2)$ is well defined since the quadratic form of the kinetic energy is positive definite. From Liouville's theorem on the conservation of the phase volume we have that $d\mu = dp_1 \wedge dp_2 \wedge d\varphi_1 \wedge d\varphi_2$ is a smooth invariant measure. This follows from the fact that the functions $Q_i$ depend only on the coordinates and the divergence of the right hand side of the above \textcolor{black}{system} for the variables $\dot \varphi_i$, $\dot p_i$ is zero. Therefore, system \eqref{eq5} has the invariant measure $d \mu = dp_1 (\varphi_1, \varphi_2, \dot \varphi_1, \dot \varphi_2) \wedge dp_2 (\varphi_1, \varphi_2, \dot \varphi_1, \dot \varphi_2) \wedge d\varphi_1 \wedge d\varphi_2$.} 

\end{proof}

{Let us give more precise definitions of integrability, which can and will be used for system \eqref{eq5}:
\begin{enumerate}
\item If the system is Hamiltonian (for instance, for $F = 0$), then it is natural to call the system integrable if the Liouville-Arnold theorem (see, for instance \cite{arnold2007mathematical}) holds: the system is integrable if there exist two independent first integrals in involution.
\item If the system is not Hamiltonian we can study the integrability in the sense of the Jacobi theorem of the last multiplier (see \cite{jacobi1884cgj} or a modern exposition of this classical theorem in \cite{kozlov2013euler}): the system is integrable if there exist two independent first integrals and a smooth invariant measure. We can also consider the integrability \textit{in a broad sense.} This notion was introduced by O. Bogoyavlenskij in \cite{bogoyavlenskij1998extended}. In our case this type of integrability means that the system has two first integrals and a two-dimensional Lie algebra of symmetries of \eqref{eq5} which preserves the first integrals. In our considerations this Lie algebra will be defined by two commuting vector fields: the trivial symmetry defined by the vector field of the system and the vector field $\frac{\partial}{\partial \varphi_2}$ which corresponds to the rotational symmetry of the system.
\end{enumerate}} 

\textcolor{black}{Let us briefly compare two definitions of integrability: the integrability in a broad sense and in the sense of the Jacobi theorem. The key difference here is that \textcolor{black}{if a system is integrable in a broad sense then the phase space is foliated into invariant tori} (if the level sets of the first integrals are compact). The proof of this fact is based on the existence of only one closed $k$-dimensional manifold which admits a $k$ commuting vector fields (this manifold is $\mathbb{T}^k$). Therefore, integrability in a broad sense is similar to  the integrability of Hamiltonian systems in the sense of the Arnold-Liouville. For instance, from the integrability in a broad it also follows that there exist action-angle variables in a neighborhood of an invariant torus. From the Jacobi theorem we do not obtain any information about the \textcolor{black}{dynamics} in the large of a considered dynamical system. We can only conclude that the corresponding ODE can be integrated by quadratures.}  

System \eqref{eq4} is not Hamiltonian. Therefore, the integrability of this system can be studied in the sense of the Jacobi theorem or \textit{in a broad sense.} However, the most natural way to define the integrability here is to apply the classical theorem of integrability for general systems of ordinary differential equations: a system in $\mathbb{R}^n$ is integrable if it has $n-1$ functionally independent first integrals. \textcolor{black}{In this case the system is also integrable in a broad sense.}

Note that a generalization of the classical Jacobi theorem presented in \cite{kozlov2013euler} \textcolor{black}{can also be} considered as a definition of integrability. However, in the paper, we will not use this general result.

\section{Two integrable cases: $k_2 = 0$ and $F = 0$ or $m_1 l_1 = m_3 l_3$}

We start our analysis of the integrability of systems \eqref{eq5} and \eqref{eq4} by considering two special integrable cases. Namely, the case where $F = 0$ and the case {in which} $m_1 l_1 = m_3 l_3$. In both cases we assume that $k_2 = 0$.

\begin{proposition}
{If $F = 0$ and $k_2 = 0$, then system \eqref{eq5} is Hamiltonian and integrable in the sense of the Liouville-Arnold theorem.}
\end{proposition}
\begin{proof}
{
If there is no follower force ($F = 0$), then the system is Hamiltonian. From $k_2 = 0$ we obtain that $\varphi_2$ is a cyclic variable. There exist two first integrals: the energy integral $H = T + \Pi$ and the cyclic integral (the integral of angular momentum)
\begin{align}
\label{eq6}
\begin{split}
&K = \frac{\partial T}{\partial \dot \varphi_2} =  m_2 \dot \varphi_2 l_2^2 +\\
&\frac{m_1}{2} \left( 2 \dot \varphi_2 l_2^2 + 2(\dot \varphi_1 + \dot \varphi_2)l_1^2 + 2l_1 l_2 \cos \varphi_1 (\dot \varphi_1 + 2 \dot \varphi_2) \right) + \\
& \frac{m_3}{2} \left( 2 \dot \varphi_2 l_2^2 + 2(\dot \varphi_1 + \dot \varphi_2) l_3^2 - 2l_2 l_3 (\dot \varphi_1 + 2 \dot \varphi_2) \cos \varphi_1 \right).
\end{split}
\end{align}
It can be easily shown that $K$ and $H$ are functionally independent and their Poisson bracket equals zero almost everywhere.}
\end{proof}

Note that system \eqref{eq4} is also integrable as a general system of ordinary differential equations (system \eqref{eq4} defined in a three-dimensional space has two first integrals), and in the sense of the Jacobi theorem of the last multiplier. The latter follows from the fact that the original system \eqref{eq5} has an invariant measure for any $F$. The density of this measure is constant if we rewrite our system in the Hamiltonian coordinates (see Lemma 1).

Another case {in which} system \eqref{eq4} is simplified and can be integrated \textcolor{black}{by quadratures} is when $m_1 l_1 = m_3 l_3$.

\begin{proposition}
{
If $m_1 l_1 = m_3 l_3$ and $k_2 = 0$, then system \eqref{eq4} is integrable in the sense of the Jacobi theorem.}
\end{proposition}
\begin{proof}
{
The system can be presented {as two} separate independent systems:
\begin{align}
\label{eq7}
\begin{split}
& \ddot \varphi_1 = \frac{F l_1 \sin \varphi_1 - k_1 \varphi_1}{(m_1 + m_2 + m_3)l_2^2},\\
& \dot v_2 = - \ddot \varphi_1 - \frac{k_1}{m_1 l_1^2 + m_3 l_3^2} \varphi_1.
\end{split}
\end{align}
The first subsystem has the invariant measure $d\mu = d\varphi_1 \wedge d v_1$, where $\dot \varphi_1 = v_1$. The proof is similar to that of Lemma 1. Therefore, system \eqref{eq7} also has the invariant measure $\rho(\varphi_1, v_1) d\varphi_1 \wedge d v_1 \wedge d v_2$. The first subsystem also has a first integral and \eqref{eq7} is integrable by the Jacobi theorem.}
\end{proof}

We have shown that systems \eqref{eq5} and \eqref{eq4} are integrable when $k_2 = 0$ and $F = 0$ or $m_1 l_1 = m_3 l_3$. These integrable cases can be explicitly studied and we prove the existence of periodic solutions that intersect the plane $\varphi_1 = 0$ at two points. The projections of the corresponding solutions onto the plane $(\varphi_1, v_1)$ are shown in Figs.~2 and 3. 

\begin{figure}[h!]
  \centering
  \includegraphics[width=\linewidth]{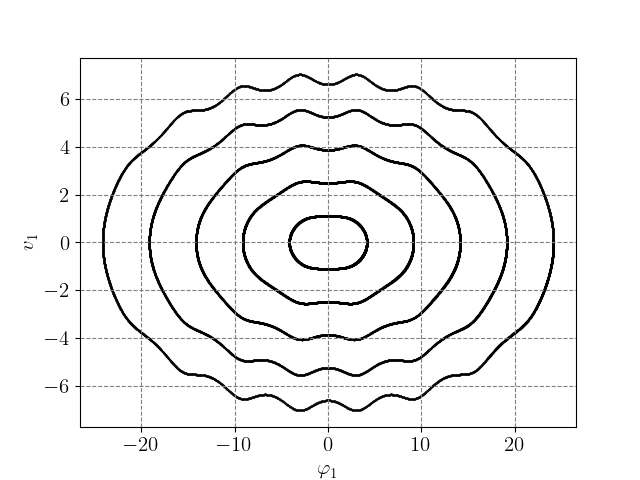}
  \label{fig:sfig2}
\caption{Periodic solutions of system \eqref{eq4} intersecting the plane $\varphi_1 = 0$ at two points for the case {in which } $F = 0$. Here $m_1 = 1$, $m_2 = 2$, $l_1 = 3$, $l_2 = 1$, $m_3 = 1$, $l_3 = 4$, $k_1 = 0.275$, $\varphi_1(0) = \pi + 1 + 5j$, $j=0,\dots, 4$, $v_1(0) = 0.1$, $v_2(0) = 0.2$.}
\end{figure}

\begin{figure}[h!]
  \centering
  \includegraphics[width=\linewidth]{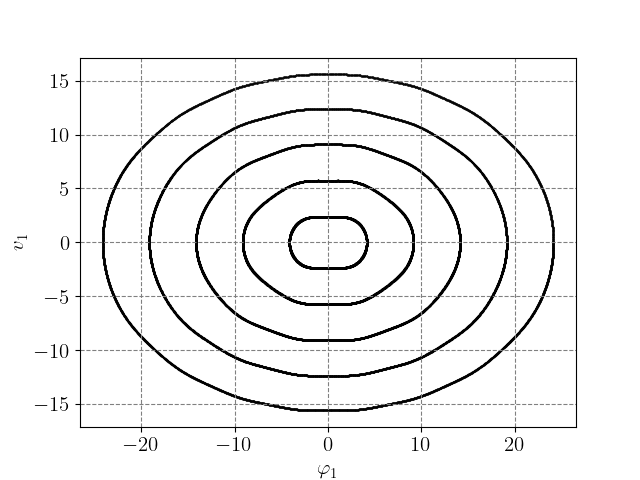}
  \label{fig:sfig1}
\caption{{Periodic solutions of system \eqref{eq4} intersecting the plane $\varphi_1 = 0$ at two points for the case in which $m_1 l_1 = m_3 l_3$. Here $F=2$, $m_1 = 1$, $m_2 = 1$, $l_1 = 2$, $l_2 = 1$, $m_3 = 2$, $l_3 = 1$, $k_1 = 1$, $\varphi_1(0) = \pi + 1 + 5j$, $j=0,\dots, 4$, $v_1(0) = 0.1$, $v_2(0) = 0.2$.}}
\end{figure}

\begin{proposition}
If $F = 0$ and $k_2 = 0$, then there exists a two-parameter family of periodic solutions of system \eqref{eq4} which intersect the plane $\varphi_1 = 0$ transversely, provided that $k_1$ is sufficiently large.
\end{proposition}
\begin{proof}
The system has two first integrals and the main idea of the proof is to show that the level sets of these integrals intersect the plane $\varphi_1 = 0$.

First, let us find an expression for $\dot\varphi_2$ from \eqref{eq6}:
$$
\dot \varphi_2 = \frac{K - V(\varphi_1) \dot \varphi_1}{U(\varphi_1)},
$$
where
\begin{align*}
\begin{split}
V(\varphi_1) &= m_1 l_1^2 + m_3 l_3^2 + (m_1 l_1 l_2 - m_3 l_2 l_3 )\cos \varphi_1,\\
U(\varphi_1) &= m_1 l_1^2 + m_2 l_2^2 + m_3 l_2^2 + m_3 l_3^2 + m_3 l_3^2\\
 &+ 2 m_1 l_1 l_2 \cos\varphi_1 - 2 m_3 l_2 l_3 \cos \varphi_1.
\end{split}
\end{align*}
The system is Hamiltonian, we can substitute the above expression for $\dot\varphi_2$ into the energy integral $H$ and obtain some function of two variables $\varphi_1$, $\dot \varphi_1$:
\begin{align*}
\begin{split}
H(\varphi_1, \dot \varphi_1) &= \frac{1}{2} l_2^2 (m_1 + m_2 + m_3) \left(\frac{K - V \dot \varphi_1}{U} \right)^2 \\
&+ \frac{1}{2} (m_1 l_1^2 + m_3 l_3^2) \left( \frac{K + (U - V)\dot \varphi_1}{U} \right)^2\\
& + \left(\frac{K - V \dot \varphi_1}{U} \right)\left( \frac{K + (U - V)\dot \varphi_1}{U} \right) \\ 
&\times  \cos \varphi_1 (m_1 l_1 l_2 - m_3 l_2 l_3) + \frac{k_1 \varphi_1^2}{2}.
\end{split}
\end{align*}

Let us prove that the function $H(\varphi_1, \dot \varphi_1)$ has an extremum on the axis $\varphi_1 = 0$. Since $U'(0) = V'(0) = 0$, for $\varphi_1 = 0$ we have
$$
\frac{\partial H}{\partial \varphi_1} = 0.
$$
From the equation
$$
\frac{\partial H}{\partial \dot\varphi_1} = 0
$$
we obtain

\begin{align*}
\begin{split}
&\dot \varphi_1 = (VK l_2^2 (m_1 + m_2 + m_3) + K(V-U)(m_1 l_1^2 \\
&+ m_3 l_3^2) + (2V - U)K (m_1 l_1 l_2 - m_3 l_2 l_3) )\\
&\times(V^2 l_2^2 (m_1 + m_2 + m_3) + (U-V)^2 (m_1 l_1^2 + m_3 l_3^2)\\
& - 2V (U-V) (m_1 l_1 l_2 - m_3 l_2 l_3))^{-1}
\end{split}
\end{align*}

We assume that $m_i > 0$, $l_i > 0$, $i = 1,2,3$. Therefore, the denominator of this fraction is always strictly positive since $(V,U) \neq (0,0)$ and it is a quadratic form of $V$ and $U$ such that the determinant of the corresponding symmetric matrix is given by 
\begin{align*}
\begin{split}
\Delta & = l_2^2(l_1^2 m_1 m_2 + l_1^2 m_1 m_3\\
& + 2l_1 l_3 m_1 m_3 + l_3^2 m_1 m_3 + l_3^2 m_2 m_3) > 0.
\end{split}
\end{align*}

Let us now show that the extremum considered is a strict local minimum. Indeed, it is not hard to see that for $\varphi_1 = 0$
$$
\frac{\partial ^2 H}{\partial \varphi_1 \partial \dot \varphi_1} = 0.
$$ 
The first minor of the Hessian matrix at the point where the gradient of $H$ equals zero has the following form:
$$
\frac{\partial ^2 H}{\partial \varphi_1^2} = c + k_1,
$$
where $c$ is a constant. If $k_1 > 0$ is large, then this value is positive. Since 
\textcolor{black}{
\begin{align*}
\begin{split}
&\frac{\partial ^2 H}{\partial \dot \varphi_1^2} = \frac{V(U-V)}{U^2} \cos \varphi_1 (m_1 l_1 l_2 - m_3 l_2 l_3) +\\
 & \frac{V^2}{U^2} l_2^2 (m_1 + m_2 + m_3) + (m_1 l_1^2 + m_3 l_3^2) \frac{(U-V)^2}{U^2}
\end{split}
\end{align*}}
 is strictly positive for all $U$ and $V$, the Hessian is positive definite.

Finally, we show that, for all values of the energy integral which is close to the critical value of the function $H$ at the point of local minimum, the trajectory is periodic and intersects the plane $\varphi_1 = 0$ transversely. If one trajectory intersects the plane $\varphi_1 = 0$ transversely and $H = h$, $K = k$ along this trajectory, then for all $K = \tilde k$, $H = \tilde h$ the corresponding periodic trajectory also intersect $\varphi_1 = 0$ at two points (if $\tilde k$ is close to $k$ and $\tilde h$ is close to $h$). This follows from the continuous dependence of the solutions from the initial conditions. 

\end{proof}

Similar propositions hold for system \eqref{eq4} if $m_1 l_1 = m_3 l_3$.

\begin{proposition}
For $k_1 > 0$ system \eqref{eq7} has a two-parameter family of periodic solutions which intersect the plane $\varphi_1 = 0$ transversely.
\end{proposition}
\begin{proof}
The first equation of system \eqref{eq7} is the Lagrange equation with the following potential energy
$$
\Pi = \frac{F l_1 \cos \varphi_1 + \frac{1}{2}k_1 \varphi_1^2}{(m_1 + m_2 + m_3)l_2^2}.
$$
If the value of total energy is sufficiently large, then the phase trajectories encircle the point of origin $\varphi_1 = 0$, $\dot \varphi_1 = 0$ and these trajectories are symmetric with respect to the coordinate axes. Therefore, $\varphi_1(t)$ is a periodic function and its mean value equals zero. We find that $v_2(t)$ is also a periodic function with the same period. From the continuous dependence on the initial data from the existence of one periodic trajectory we obtain a two-parameter family of such trajectories.
\end{proof}

\begin{proposition}
For $k_1 = 0$ and $F < 0$ system \eqref{eq7} {has} a two-parameter family of periodic solutions which intersect the plane $\varphi_1 = 0$ transversely.
\end{proposition}
\begin{proof}
For the first equation of the system, the point $\varphi_1 = 0$, $\dot \varphi_1 = 0$ is a local minimum. Since the first equation of system \eqref{eq7} is a Hamiltonian system, this equilibrium is of center type, i.e., all trajectories which are close to this equilibrium are periodic. Again, the function $v_2(t)$ has the same period as the corresponding periodic solution 
\end{proof}

The periodic solutions obtained in Propositions 4--6 are orbitally stable, but their periods are different and these solutions are Lyapunov unstable.

Note that the Ziegler pendulum can be considered as a Hamiltonian system for some values of the parameters \cite{kozlov2021integrals}. However, for the case $k_2 = 0$ it follows from the conditions presented in \cite{kozlov2021integrals} that $m_1 l_1 = m_3 l_3$ and $k_1 = 0$. Therefore, from \eqref{eq5} we obtain a trivial system, which is less general than \eqref{eq7}.

\section{Integrability}

In the proofs of Propositions 4--6 we used the following fact: if there is a periodic trajectory which intersects the plane $\varphi_1 = 0$ transversely, then there exists a two-parameter family of such {solutions}. We used the continuous dependence on the initial data, but the same holds if we slightly vary the right-hand sides of the corresponding equations. In other words, there are two-parameter families of periodic solutions in system \eqref{eq4} if $k_2 = 0$ and $|F|$ or $|m_1 l_1 - m_3 l_3|$ are sufficiently small. This fact plays a key role in our proof of the integrability.

For the case where $k_2 = 0$ it is natural to call the original system integrable if the reduced system \eqref{eq4} has two first integrals. \textcolor{black}{In particular, in this case system \eqref{eq5} is integrable in a broad sense.} In the general case, for which it is also allowed that $k_2 \ne 0$, it is also natural to call the system with two first integrals integrable. For any $k_2$ we have a system defined on $\mathbb{T}^2 \times \mathbb{R}^2$ and this system has an invariant measure. Therefore, this system is integrable if it has two first integrals, similarly to Hamiltonian systems with two degrees of freedom and to the case $k_2 = 0$.

Let us prove that system \eqref{eq4} (where $k_2 = 0$) remains integrable for $|F|$ or $|m_1 l_1 - m_3 l_3|$ sufficiently small. The key step in the proof is given by the following lemma.

\begin{lemma}
Let the functions $\varphi_1(t)$, $v_1(t)$, $v_2(t)$ be a solution of system \eqref{eq4}. Then the functions $-\varphi_1(-t)$, $v_1(-t)$, $v_2(-t)$ are also a solution.
\end{lemma}
\begin{proof}

The proof follows from the substitution of given functions into the system. Terms $A_{11}$, $A_{12}$, $A_{22}$ do not change after the substitution $\varphi_1 \mapsto -\varphi_1$. Terms $r_1$ and $r_2$ change their signs. At the same time, the left-hand side also changes its sign.
\end{proof}
From the lemma we obtain the following
\begin{proposition}
Let $\varphi_1(t)$, $v_1(t)$, $v_2(t)$  be a solution  of \eqref{eq4}which intersects the plane $\varphi_1 = 0$ in two points. Then this solution is periodic and intersects the plane $\varphi_1 = 0$ transversely.
\end{proposition}
\begin{proof}
A periodic solution can be obtained from the segment between the two points using symmetry with respect to the plane $\varphi_1 = 0$.
\end{proof}
Finally, we obtain the following three results on the existence of families of periodic solutions for system \eqref{eq4}.
\begin{proposition}
Let $k_2 = 0$. There exists $K_1$ such that for any $k_1 > K_1$ there exists $\varepsilon = \varepsilon(k_1)$ and for all $|F| < \varepsilon$ system \eqref{eq4} has a two-parameter family of periodic solutions.
\end{proposition}
\begin{proof}
 For $F = 0$ there exist $K_1$ such that for any $k_1 > K_1$ there exists a two-parameter family of periodic solutions of \eqref{eq4}. Moreover, each of these solutions intersects the plane $\varphi_1 = 0$ transversely. Therefore, if $|F|$ is small, then the right-hand side of \eqref{eq4} is close to the corresponding right hand side for $F = 0$ and the system still has a two-parameter family of periodic solutions. 
\end{proof}

The proofs of the following two propositions are generally the same and are based on the existence of families of solutions which intersect the plane $\varphi_1 = 0$ transversely for $m_1 l_1 - m_3 l_3 = 0$. From Proposition 8 we find that these families of solutions exist for small values of $|m_1 l_1 - m_3 l_3|$.

\begin{proposition}
Let $k_1 > 0$ and $k_2 = 0$, then there exists $\varepsilon = \varepsilon(k_1)$ such that for all $|m_1 l_1 - m_3 l_3| < \varepsilon$ system \eqref{eq4} has a two-parameter family of periodic solutions. 
\end{proposition}
\begin{proposition}
Let $k_1 = 0$, $k_2 = 0$ and $F < 0$, then there exists $\varepsilon = \varepsilon(F)$ such that for all $|m_1 l_1 - m_3 l_3| < \varepsilon$ system \eqref{eq4} has a two-parameter family of periodic solutions. 
\end{proposition}

{Let us now discuss the integrability of system \eqref{eq4}. First, let us construct --- to a large degree explicitly --- the first integrals for this system. We will define these first integrals in a neighborhood of a periodic trajectory which intersects the plane $\varphi_1 = 0$.}

{Let $(v_1^0, v_2^0)$ be a local coordinates on the plane $\varphi_1 = 0$ in a neighborhood of the point of intersection of some periodic solution of \eqref{eq4} with the plane considered. For any point in a open neighborhood (in $\mathbb{R}^3$) of the periodic trajectory we define two functions $F_1(\varphi_1, v_1, v_2)$ and $F_2(\varphi_1, v_1, v_2)$ as follows. We put $F_1(\varphi_1, v_1, v_2) = v_1^0$ and $F_2(\varphi_1, v_1, v_2) = v_2^0$ if for some $t$ we have $v_1(t; 0, v_1^0, v_2^0) = v_1$, $ v_2(t; 0, v_1^0, v_2^0) = v_2$, $\varphi_1(t; 0, v_1^0, v_2^0) = \varphi_1$, where $\varphi_1(t;\varphi_1^0,v_1^0,v_2^0)$, $v_i(t;\varphi_1^0,v_1^0,v_2^0)$, $i=1,2$ are the components of solution of \eqref{eq4} with initial conditions $(\varphi_1^0,v_1^0,v_2^0)$.}

{In other words, the values of the first integrals $F_1$, $F_2$ in a neighborhood of a periodic trajectory are defined by the coordinates $v_1^0$ and $v_2^0$ of the point of intersection of the corresponding trajectory with the plane $\varphi_1 = 0$.}

{Since the flow of \eqref{eq4} defines a $C^\infty$-smooth diffeomorphism, it is not hard to show that $F_i$ are $C^\infty$-smooth functions. Two $1$-forms $dF_1$ and $dF_2$ are linearly independent in a neighborhood of the periodic solution considered: since $(v_1^0, v_2^0)$ are local coordinates on the plane $\varphi_1 = 0$, then the $dF_1$ and $dF_2$ are independent (by construction) at the point of intersection of the periodic trajectory with the plane $\varphi_1 = 0$, the flow of our system defines a diffeomorphism and this diffeomorphism preserves the linear independence of the $1$-forms (or corresponding vectors) along the trajectory.}

{Any function of $F_1$ and $F_2$ is also a first integral and we can construct infinitely many first integrals (of course, these integrals are not independent of $F_1$, $F_2$). For instance, we can construct $C^k$-smooth first integrals for any $k \geqslant 0$ as a $C^k$-smooth function of $F_1$ and $F_2$.} \textcolor{black}{Since the right-hand side of the considered system is analytic in a neighborhood of the periodic solution, then it is also possible to construct real analytic first integrals $F_1$ and $F_2$.} 

\begin{proposition}
{ Let $k_2 = 0$. There exists $K_1$ such that for any $k_1 > K_1$ there exists $\varepsilon = \varepsilon(k_1)$ and for all $|F| < \varepsilon$ system \eqref{eq4} is integrable: there are two functionally independent smooth first integrals. System \eqref{eq5} is integrable in the sense of the Jacobi theorem and is integrable \textit{in a broad sense} }. 
\end{proposition}

\begin{proof}
{We have already proved the existence of an invariant measure, therefore the systems are integrable in the first two senses. The integrability \textit{in a broad sense} follows from the existence of two first integrals and two commuting vector fields in $\mathbb{R}^4$: the vector field of system \eqref{eq5} and $\frac{\partial}{\partial \varphi_2}$. The constructed first integrals $F_1$ and $F_2$ do not depend on $\varphi_2$ and are also first integrals for the vector field $\frac{\partial}{\partial \varphi_2}$. }
\end{proof}

Similarly, we obtain the following proposition.
\begin{proposition}
Let $k_2 = 0$. Let $k_1 > 0$ or $k_1 = 0$ and $F < 0$, then there exists $\varepsilon = \varepsilon(F)$ such that for all $|m_1 l_1 - m_3 l_3| < \varepsilon$ systems \eqref{eq5} and \eqref{eq4} are integrable.
\end{proposition}

Here the integrability is understood in the same sense as in Proposition 12.

\begin{figure}
  \centering
  \includegraphics[width=\linewidth]{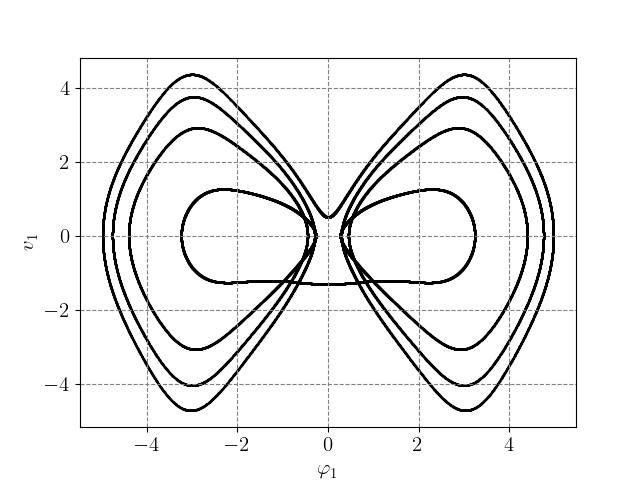}
  \caption{{An example of a periodic solution of system \eqref{eq4} for $F \ne 0$ and $m_1 l_1 \ne m_3 l_3$. Projection onto the $(\varphi_1, v_1)$-plane. Here $F=6$, $m_1 = 1$, $m_2 = 2$, $l_1 = 1$, $l_2 = 4$, $m_3 = 2$, $l_3 = 2.3$, $k_1 = 3$, $\varphi_1(0) = 0$, $v_1(0) = 0.5$, $v_2(0) = 2.25$.}}
  \label{fig:sfig2}
\end{figure}%

\begin{figure}
  \centering
  \includegraphics[width=\linewidth]{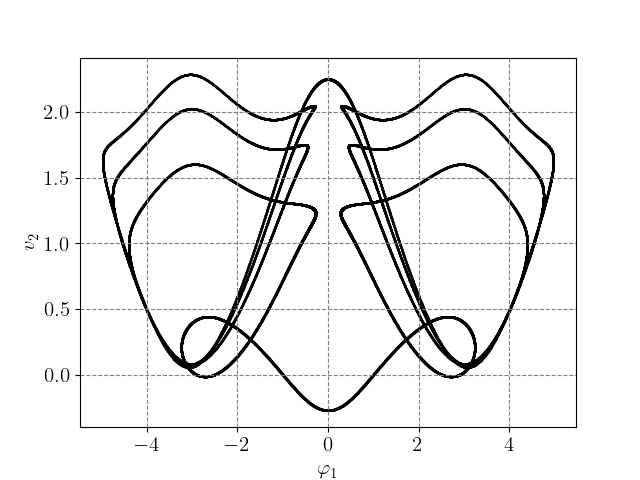}
  \caption{{An example of a periodic solution of system \eqref{eq4} for $F \ne 0$ and $m_1 l_1 \ne m_3 l_3$. Projection onto the $(\varphi_1, v_2)$-plane. Here $F=6$, $m_1 = 1$, $m_2 = 2$, $l_1 = 1$, $l_2 = 4$, $m_3 = 2$, $l_3 = 2.3$, $k_1 = 3$, $\varphi_1(0) = 0$, $v_1(0) = 0.5$, $v_2(0) = 2.25$.}}
  \label{fig:sfig2}
\end{figure}%

The natural question of whether or not it is possible to lift the two first integrals $F_1$ and $F_2$ globally to the whole three-dimensional space with coordinates $\varphi_1, v_1, v_2$ is expected to have a negative answer. It follows from the numerical results which show that in a general case for some initial conditions there exist non-periodic trajectories lying on a two-dimensional surface. Therefore, there is numerical evidence that there are less than two independent first integrals. Let us however note that the periodic solutions which intersect the plane $\varphi_1 = 0$ can be relatively complex compared to that presented in Figs.~2 and 3. An example of such a solution is shown in Figs.~4 and 5.

The mechanism for the onset of non-periodic trajectories is as follows (see Fig.~6). Suppose that for $F = 0$ we have a family of periodic solutions, if $|F|$ is sufficiently small, then all these solutions still intersect the plane $\varphi_1 = 0$ at two points. Therefore, we obtain a family of periodic solutions. When $F$ becomes sufficiently large, some solutions of this family do not intersect the plane $\varphi_1 = 0$ and may become non-periodic. This non-periodic trajectories lie on a two-dimensional surface in $\mathbb{R}^3$. Therefore, we have no more than one first integral of the system. This breaks the integrability (or super-integrability) of the system. Figure~6  shows a family of solutions starting at the same point. The only parameter of the system that varies from solution to solution is the magnitude of force $F$. For small $F$ all solutions intersect the plane $\varphi_1 = 0$, but a non-periodic solution appears when $F$ becomes larger than some critical value. This solution does not intersect the plane considered.

Note that the above conditions for the existence of periodic solutions for system \eqref{eq4} (for $F = 0$) and \eqref{eq7} can be weakened and the families of periodic solutions are much wider. We present simple conditions that can be obtained without cumbersome calculations and detailed estimates.

\begin{figure}[h!]
  \centering
  \includegraphics[width=\linewidth]{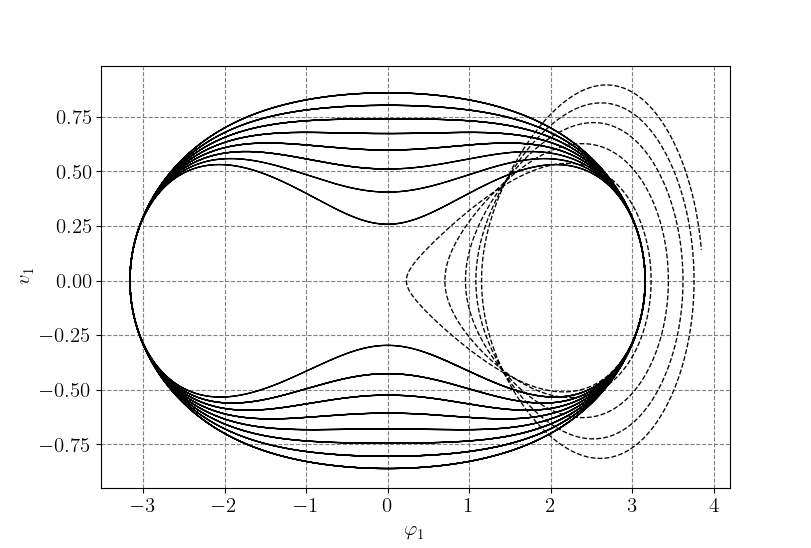}
\caption{The mechanism for the onset of a non-periodic solution. The trajectory of the non-periodic solution is dotted ($F=0.8$). Here $F=0.1j$, $j=0,\dots,8$, $m_1 = 1$, $m_2 = 1$, $l_1 = 3$, $l_2 = 1$, $m_3 = 1$, $l_3 = 4$, $k_1 = 0.275$, $\varphi_1(0) = \pi$, $v_1(0) = 0.1$, $v_2(0) = 0$.}
\end{figure}

For the integrability in a four-dimensional space with an invariant measure, the most natural definition of the integrability stems from the Jacobi theorem: the system is integrable if it has two independent first integrals.

For any periodic solution in the reduced space with coordinates $\varphi_1,v_1, v_2$, from the independent equation
$$
\dot \varphi_2 = v_2,
$$
it follows that in the four-dimensional space $\varphi_1,\varphi_2,v_1, v_2$ this solution lies on a two-dimensional surface (the angular variables are considered modulo $2\pi$). If the solution is not periodic in the reduced space $\varphi_1,v_1, v_2$, then its projection from the four-dimensional space onto a three-dimensional space with, for instance, coordinates $\varphi_1, \varphi_2, v_1$ can densely fill three-dimensional domains. 

In the former case, when we have periodic solutions, we have two first integrals. In the latter case, when solutions in $\mathbb{R}^3$ are not periodic, there may be only one or no first integrals. As an illustration (Figs.~7 and 8) we present the sections formed by the intersection of projections of two groups of solutions onto the coordinate space  $\varphi_1, \varphi_2, v_1$ with the plane $\varphi_2 = \pi/2$: in the first case the solutions are periodic in the space with coordinates $\varphi_1, v_1, v_2$, in the second case the corresponding solution is not periodic in this space. The parameters of the system are the same in the both cases, yet the initial conditions vary.

\begin{figure}
  \centering
  \includegraphics[width=\linewidth]{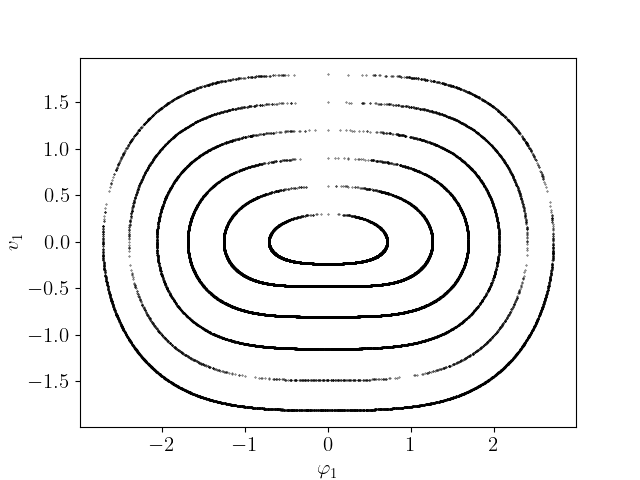}
  \caption{Regular behavior of solutions of system \eqref{eq5}. The sections of six independent solutions are presented, each solution lies on an invariant torus. Here $F=2$, $m_1 = 1$, $m_2 = 1$, $l_1 = 1$, $l_2 = 1$, $m_3 = 1.1$, $l_3 = 1$, $k_1 = 1$, $\varphi_1(0) = 0$, $\varphi_2(0) = 0$, $v_1(0) = 0.3j$, $j=0,\dots, 6$, $v_2(0) = 0$.}
  \label{fig:sfig2}
\end{figure}%

\begin{figure}
  \centering
  \includegraphics[width=\linewidth]{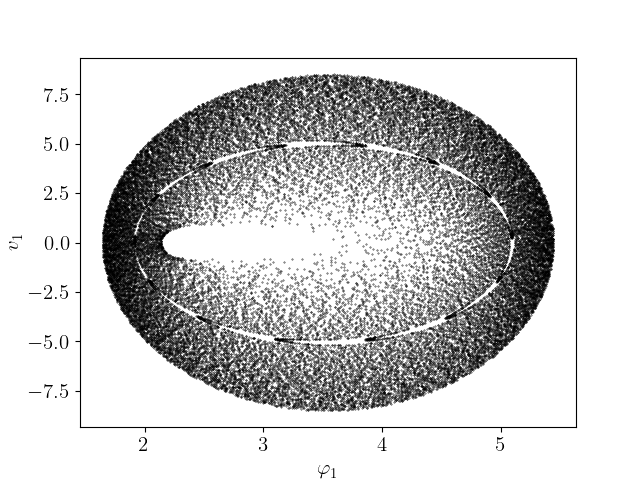}
  \caption{Irregular behavior of one solution of system \eqref{eq5}. The section of a single solution which does not belong to an invariant two-dimensional torus is shown. Here $F=2$, $m_1 = 1$, $m_2 = 1$, $l_1 = 1$, $l_2 = 1$, $m_3 = 1.1$, $l_3 = 1$, $k_1 = 1$, $\varphi_1(0) = \pi$, $\varphi_2(0) = 0$, $v_1(0) = 0.2$, $v_2(0) = -10$.}
  \label{fig:sfig1}
\end{figure}

In the conclusion of the section, we consider a sufficient generalization of system \eqref{eq4} which is also integrable. Indeed, the key property which is used in the proves of integrability is the symmetry of the original system. Hence, if we consider perturbations relatively small and satisfying the properties of symmetry, {then we again obtain a locally integrable system.} For instance, let us consider the following system (here we assume that $k_2 = 0$) 
\begin{align}
\label{eq8}
\begin{split}
& \dot \varphi_1 = v_1,\\
& \dot v_1 = \left( r_1 - \frac{A_{12}}{A_{22}} r_2 \right)\left( A_{11} - \frac{A_{12} A_{21}}{A_{22}} \right)^{-1} + f_1,\\
& \dot v_2 = \left( r_2 - \frac{A_{21}}{A_{11}} r_1 \right)\left( A_{22} - \frac{A_{12} A_{21}}{A_{11}} \right)^{-1} + f_2,
\end{split}
\end{align}
where $f_i = f_i(\varphi_1, v_1, v_2)$ and $f_i(-\varphi_1, v_1, v_2) = -f_i(\varphi_1, v_1, v_2)$, $i=1,2$ are sufficiently small. Then this system is integrable in the above senses. For instance, a family of periodic solutions exists for system \eqref{eq8} if we put $f_1 = -\alpha \varphi_1 \sin(v_2)$, $f_2 = \alpha \varphi_1 \sin(v_1)$, for sufficiently small values of $\alpha$. At the same time, there may exist asymptotically stable equilibria in this system. Therefore, there is no smooth invariant measure in this system and there cannot exist two first integrals which are defined globally in the whole phase space.

\section{Case $k_2 \ne 0$}

In the above considerations we assume that the stiffness of the spring located at the fixed point $O$ equals zero. This allows us to study the reduced system with the three-dimensional phase space $\mathbb{S} \times \mathbb{R}^2$.

When $k_2 \ne 0$ we have to consider the system defined on $\mathbb{T}^2 \times \mathbb{R}^2$, and this system has an invariant measure. Based on the results from KAM theory for the reversible systems (for details, see \cite{sevryuk2006reversible}), one can expect the coexistence of regions with regular and chaotic behavior of the solutions. 

A detailed numerical study of this case is beyond the scope of this paper. We only present some results showing that for the Ziegler pendulum the regular dynamics can coexist with the irregular behavior of trajectories.

After rescaling time by
$$
t \mapsto t \cdot (A_{11}A_{22} - A_{12}A_{21}) 
$$
system \eqref{eq4} can be represented in the following form (here we also have the equation for $\dot \varphi_2$)
\begin{align}
\label{eq9}
\begin{split}
& \dot \varphi_1 = v_1 (A_{11}A_{22} - A_{12}A_{21}),\\
& \dot \varphi_2 = v_2 (A_{11}A_{22} - A_{12}A_{21}),\\
& \dot v_1 = A_{22} r_1 - A_{12} r_2,\\
& \dot v_2 = A_{11}r_2 - A_{12} r_1,
\end{split}
\end{align}
where we have used the same notations as above.

One of the most popular methods of numerical study of chaos is $mLCE$ (a simple method based on the estimation of the maximum Lyapunov characteristic exponent). A detailed explanation of this approach can be found in \cite{benettin1980lyapunov,lichtenberg2013regular,skokos2010lyapunov}.

For a given solution, the Lyapunov characteristic exponents can be calculated as follows:
$$
\chi(X(t), x) = \limsup\limits_{t \to + \infty} \frac{1}{t} \ln \frac{\|X(t) x(0)\|}{\|x(0)\|},
$$
where $X(t)$ is the fundamental matrix of the linearized system for system \eqref{eq9}. The system is linearized along the  solution considered. The evolution of vector $x(t)$ is described by the following linear system:
$$
\dot x = A(t) x.
$$
For our system, we have

\begin{align*}
A(t) = 
\left(\begin{matrix}
v_1A_{11} \displaystyle{\frac{\partial A_{22}}{\partial \varphi_1}} - 2 v_1  A_{12} \frac{\partial A_{12}}{\partial \varphi_1}\\
v_2 A_{11} \displaystyle{\frac{\partial A_{22}}{\partial \varphi_1}} - 2 v_2 A_{12} \frac{\partial A_{12}}{\partial \varphi_1} \\
\displaystyle{ \frac{\partial A_{22}}{\partial \varphi_1} r_1 + A_{22} \frac{\partial r_1}{\partial \varphi_1} - \frac{\partial A_{12}}{\partial \varphi_1} r_2 - A_{12} \frac{\partial r_2}{\partial \varphi_1}} \\
\displaystyle{ \frac{\partial A_{22}}{\partial \varphi_1} r_1 + A_{22} \frac{\partial r_1}{\partial \varphi_1} - \frac{\partial A_{12}}{\partial \varphi_1} r_2 - A_{12} \frac{\partial r_2}{\partial \varphi_1}} 
\end{matrix}\right.
\end{align*}

\begin{align*}
\left.
\begin{matrix}
0 & A_{11}A_{22} - A_{12}A_{21} & 0\\
0 & 0 & A_{11}A_{22} - A_{12}A_{21}\\
- A_{12} \displaystyle{\frac{\partial r_2}{\partial \varphi_2}} & A_{22} \displaystyle{\frac{\partial r_1}{\partial v_1}} - A_{12} \displaystyle{\frac{\partial r_2}{\partial v_1}} & A_{22} \displaystyle{\frac{\partial r_1}{\partial v_1}} - A_{12} \displaystyle{\frac{\partial r_2}{\partial v_1}}\\
A_{11} \displaystyle{\frac{\partial r_2}{\partial \varphi_2}} & A_{22} \displaystyle{\frac{\partial r_1}{\partial v_1}} - A_{12} \displaystyle{\frac{\partial r_2}{\partial v_1}} & A_{22} \displaystyle{\frac{\partial r_1}{\partial v_1}} - A_{12} \displaystyle{\frac{\partial r_2}{\partial v_1}}
\end{matrix}\right)
\end{align*}

In practice, the following equation is used to obtain the maximum Lyapunov characteristic exponent
$$
\chi(X(t)) = \lim\limits_{t \to + \infty} \frac{1}{t} \ln \frac{\|X(t) x(0)\|}{\|x(0)\|}.
$$
It is known that for almost all initial vectors $x(0)$ and for a wide range of matrices $X(t)$ this equation gives the maximum Lyapunov characteristic exponent.

Let us consider a system with the following values of the mechanical parameters:
\begin{equation*}
  \begin{split}
    &m_1 = 1,\\
    &m_2 = 1,\\
    &m_3 = 3/2,
  \end{split}
\quad
  \begin{split}
    &l_1 = 1,\\
    &l_2 = 1,\\
    &l_3 = 1,
  \end{split}
 \quad
  \begin{split}
    &k_1 = 1,\\
    &k_2 = 1,\\
    &F = 2,
  \end{split} 
\end{equation*}
and two solutions with initial conditions $(\pi, 0, 1/10, 1/10)$ and $(\pi, 0, 20, 20)$ (the order of the coordinates is as follows: $\varphi_1, \varphi_2, v_1, v_2$). It was found numerically that $mLCE = 0$ for the first trajectory and $mLCE > 0$ (a property of a chaotic trajectory) for the second one. 

Regularity and irregularity of solutions can also be seen in the `Poincar{\'e} section' of a family of solutions, from which it can also be seen that regular dynamics coexists with chaotic dynamics (Fig. 9). We consider ten initial conditions $(\pi, 0, 1/10, (n + 1)/10)$, $0 \leqslant n < 10$, $n \in \mathbb{Z}$. For each corresponding solution we consider its projection onto the three-dimensional space $(\varphi_1, v_1, v_2)$, and then the section formed by the intersection of this projection with the plane $v_1 = 0$.

\begin{figure*}
  \centering
  \includegraphics[width=1.0\linewidth]{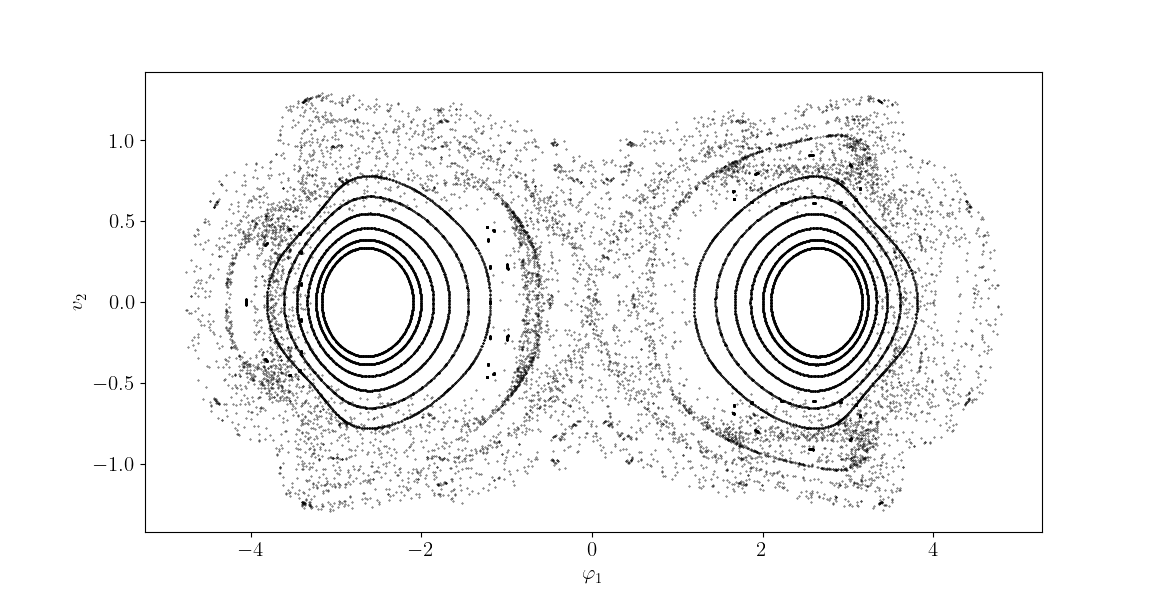}
\caption{Regular and chaotic solutions.}
\end{figure*}

The resulting sets of points are not the Poincar\'e section in the usual sense since we consider the section of the projections of solutions. In particular, it is possible that there exists more than one solution which intersects the plane $v_1 = 0$ at a given point. The resulting sections are shown in Fig.~9. It can be seen that for the relatively small initial velocity the solutions are regular (lie on a two-dimensional surface) and become chaotic as the velocity becomes larger.

\section{Conclusion}

In this paper, the dynamics of a planar double pendulum in the presence of a follower force has been studied. The system has been considered in the case {in which} the governing system of equations can be reduced to a system in a three-dimensional space, and also in the general case in which the system has been considered in a four-dimensional space. The following results have been obtained:
\begin{enumerate}
\item It is proved that, in the case of zero stiffness of the spring located at the point of suspension, there exist two-parametric families of periodic solutions {for some values of the mechanical parameters. In particular, this means the integrability of the system. } 
\item It is shown numerically that the periodic solutions do not fill the whole phase space. In particular, some solutions lie on two-dimensional surfaces and the system does not have two independent first integrals defined globally in the space with coordinates $\varphi_1, v_1, v_2$. Moreover, the system is not integrable \textcolor{black}{in a broad sense} in the whole phase space if we consider the dynamics in the original (unreduced) phase space $\mathbb{T}^2 \times \mathbb{R}^2$.
\item It is shown {numerically } that, in the general case of the Ziegler pendulum in which $k_2 \ne 0$, regular and chaotic dynamics coexist in the phase space.
\end{enumerate}

Let us briefly discuss the above results and outline some possible directions of further research. First, we note that we infer integrability of the system without explicitly presenting the required smooth functions of the first integrals. We only prove that there exist periodic solutions for the system and, therefore, the system can be integrable locally in some region of the phase space. {However, this local integrability should not be confused with the local existence of first integrals which is guaranteed by the flow-box theorem: in contrast to the flow-box theorem, we prove the existence of topologically non-trivial trajectories of the systems considered, the periodic solutions, and these solutions fully cover the region where the system is integrable. } Moreover, there is strong evidence that these local first integrals cannot be lifted to the whole phase space. {Note that some similar effects have been discovered recently in the dynamics of the so-called omnidisk \cite{kilin2023dynamics}.

An important property of the circulatory system \eqref{eq5}, which cannot be seen in non-integrable Hamiltonian systems, is that there are regions in the phase space of system \eqref{eq5} such that \textit{all} invariant tori are not destroyed when we add a perturbation to the system. For Hamiltonian systems it is typical that resonant tori are destroyed by perturbations. Since resonant tori are dense in the phase space, we can usually conclude that the system becomes non-integrable.

At the same time, the coexistence of the regular and the chaotic behavior of solutions is typical for Hamiltonian systems (see, for instance, \cite{arnold2007mathematical,kozlov2012symmetries}). Nevertheless, we have shown that trajectories with irregular behavior in circulatory systems (the Ziegler pendulum for $k_2 = 0$) may appear in a different manner comparing to Hamiltonian systems. The existence of chaotic regions in analytic Hamiltonian systems means that the system is not integrable, but the existence of a \textit{local} set of independent first integrals is still possible if we have less regular behavior of solutions in a circulatory system.

\section*{Acknowledgments}

The author is grateful to Valery Kozlov for his invaluable advice and useful suggestions.  

\bibliographystyle{unsrt}
\bibliography{sn-article}

\section*{Declarations}

The work was supported by the Russian Science Foundation (Project No. 19-71-30012). The author have no relevant financial or non-financial interests to disclose. The manuscript has no associated data.

\end{document}